\documentclass[11pt,a4paper]{amsart}
\usepackage{amsmath,amssymb,amscd,amsthm,amsxtra,array}
\usepackage[usenames,dvipsnames]{xcolor}
\usepackage[
colorlinks=true,
linkcolor=blue,
citecolor=blue,
urlcolor=blue,
pagebackref,
]{hyperref}

\newtheorem{theorem}{Theorem}[section]
\newtheorem{lemma}[theorem]{Lemma}

\theoremstyle{definition}
\newtheorem{definition}[theorem]{Definition}

\newtheorem{example}[theorem]{Example}

\theoremstyle{remark}
\newtheorem{remark}[theorem]{Remark}
\numberwithin{equation}{section}
\DeclareMathOperator*{\esssup}{ess\,sup}

\def\Xint#1{\mathchoice
  {\XXint\displaystyle\textstyle{#1}}%
  {\XXint\textstyle\scriptstyle{#1}}%
  {\XXint\scriptstyle\scriptscriptstyle{#1}}%
  {\XXint\scriptscriptstyle\scriptscriptstyle{#1}}%
  \!\int}
\def\XXint#1#2#3{{\setbox0=\hbox{$#1{#2#3}{\int}$}
    \vcenter{\hbox{$#2#3$}}\kern-.5\wd0}}

\def\avgint{\Xint-}

\numberwithin{equation}{section}

\begin{document}
\title[Buckley's theorem in LCA groups ]{Improved  Buckley's theorem on LCA groups}

\author{Victoria Paternostro}
\author{ Ezequiel Rela}
\address{Departamento de Matem\'atica,
Facultad de Ciencias Exactas y Naturales,
Universidad de Buenos Aires, Ciudad Universitaria
Pabell\'on I, Buenos Aires 1428 Capital Federal Argentina} \email{vpater@dm.uba.ar, erela@dm.uba.ar}

\thanks{ Both authors are partially supported by grants UBACyT 20020130100403BA,  CONICET-PIP 11220150100355 and PICT 2014-1480.}



\begin{abstract}
We present sharp quantitative weighted norm inequalities for the Hardy-Littlewood maximal function in the context of Locally Compact Abelian Groups, obtaining an improved version of the so-called Buckley's Theorem.  On the way, we 
prove a precise reverse H\"older inequality for Muckenhoupt $A_\infty$ weights and provide a valid version of the ``open property" for Muckenhoupt $A_p$ weights.
\end{abstract}

\maketitle

\section{Introduction and main results}

The study of weighted norm inequalities for maximal type operators is one of the central topics in harmonic analysis that began with the celebrated theorem of Muckenhaupt \cite{Muckenhoupt:Ap}. It    states that the class of weights (nonnegative locally integrable functions) characterizing the boundedness of the Hardy-Littlewood maximal function $M$ on the weighted space $L^p(\mathbb{R}^n,w dx)$ is the so-called Muckenhaupt $A_p$ class (see below for the precise definitions). It is important to remark that Muckenhaupt's result is qualitative, that is, it does not provide any precise information of how the operator norm of $M$ depends on the underlying weight in $w\in A_p$. The first quantitative  result on the boundeness  for the maximal function in $\mathbb{R}^n$ dates back to the 90's, is due to Buckley \cite{Buckley} and  gives the best possible power dependence on the $A_p$ constant $[w]_{A_p}$. More precisely, Buckley proved that 
\begin{equation}\label{eq:Buckley}
\|M\|_{L^p(\mathbb{R}^n,w dx)\to L^p(\mathbb{R}^n,w dx)}\le C [w]^{\frac{1}{p-1}}_{A_p}, \qquad 1<p<\infty.
\end{equation}
Recently a simpler and elegant proof was presented by Lerner \cite{Lerner-Elementary} who used  a very clever argument composing weighted versions of the maximal function. Later, finer improvements have been found. In particular, there is in \cite{HPR1} a sharp mixed bound valid in the context of spaces of homogeneous type.

Our purpose here is to obtain sharp quantitative norm estimates in the context of Locally Compact Abelian groups (LCA groups).
The modern approach to this problem is to use a sharp version of the reverse H\"older inequality (RHI) with a precise quantitative expression for the exponent to derive a proper open property for the $A_p$ claseses. Then an interpolation type argument allows to prove the desired bound.

In the rest of the introduction we first described in details the context where we will work in and then properly state the results that we will prove. 
 
\subsection{Muckenhoupt  weights and maximal function on LCA groups}

In the euclidean setting the standard way to introduce $A_p$ weights is by considering averages over cubes, balls or more general families of convex sets. In any case, the family is built using some specific metric. In our context of LCA groups  we lack of such concept. However there are many LCA groups where  we do have the possibility of consider a family of \emph{base sets} satisfying the other fundamental property of the basis of cubes or balls: any point has a family of decreasing base sets shrinking to it and, in addition, the whole space can be covered by the increasing union of such family. 

In order to properly defined the $A_p$ classes let us fix an LCA group $G$ with a measure $\mu$ that is inner regular and such that $\mu(K)<\infty$ for every compact set $K\subset G$.  Notice that $\mu$ does need to be the Haar measure because we do not assume $\mu$ to be translation invariant. The reader can find a comprehensive treatment of Harmonic Analysis on LCA groups in \cite{HR-vol2, HR-vol1, rudinLCA}. 
The general assumption on the group will be that it admits a sequence of neighborhoods of $0$ with certain properties that we described in the next definition (cf.  \cite[Section 2.1] {EG77}). 

\begin{definition}
A collection $\{U_i\}_{i\in\mathbb{Z}}$ is a covering family for $G$ if 
\begin{enumerate}
\item $\{U_i\}_{i\in\mathbb{Z}}$  is an increasing base of relatively compact neighborhoods of $0$,	$\bigcup_{i\in\mathbb{Z}}U_i=G$ and $\bigcap_iU_i=\{0\}$.
\item There exists a positive constant $D\ge 1$ and an increasing function $\theta:\mathbb{Z}\to\mathbb{Z}$ 
such that for any $i\in \mathbb{Z}$ and any $x\in G$
\begin{itemize}
\item $i\le\theta(i)$
\item $U_i - U_i\subset U_{\theta(i)}$
\item $\mu(x+ U_{\theta(i)})\le D\mu (x+U_i)$.
\end{itemize}
\end{enumerate}
We will refer to the third condition as the \emph{doubling} property of the measure $\mu$ with respect to $\theta$ and we will call $D$ the doubling constant.
In the case of $\mathbb{R}^n$ equipped with the natural metric and measure, we can consider the family of dyadic cubes of sidelength $2^i$ or the euclidean balls $B(x,2^i)$ for $i\in \mathbb{Z}$. The doubling constant of the Lebesgue measure in this context is $2^n$ and the function $\theta$ can be taken to be $\theta(i)=i+1$. Therefore, the intuition here is that the index $i$ in the above definition can be seen as a sort of \emph{radius} or \emph{size} of the given set $U_i$. 

For each $x\in G$, the set $x+U_i$ will be called \emph{base set} and the collection of all base sets will be denoted by 
\begin{equation}\label{eq:basesets}
\mathcal{B}:=\left\{x+U_i : x\in G,  i\in \mathbb{Z}\right\}.
\end{equation}
\end{definition}

The notion of base sets allows to define a direct analogue of the Hardy Littlewood maximal function:
\begin{equation}\label{eq:Maximal}
Mf(x) = \sup_{x\in U\in\mathcal{B} } \avgint_{U} |f|\ d\mu :=\sup_{x\in U\in\mathcal{B} } \frac1{\mu(U)} \int_{U} |f|\ d\mu. 
\end{equation}

As we already mentioned, our purpose here is to prove sharp weighted norm inequalities for this operator in $L^p(G,w d\mu)$, where $w$ is a weight on $G$ . Firstly, recall that the celebrated Muckenhoupt's Theorem asserts that the class of weights characterizing the boundedness of $M$ on $L^p(\mathbb{R}^n,w dx)$, $p>1$ is the Muckenhoupt $A_p$ class defined in $\mathbb{R}^n$ by

\begin{equation}\label{eq:A_p-Rn}
[w]_{A_p(\mathbb{R}^n,dx)}:=\sup_{Q} \left(\avgint_Q w\,d\mu\right) \left(\avgint_Q w^{1-p'}\,d\mu\right)^{p-1} <\infty.
\end{equation}

Here $p'$ denotes the conjugate exponent of $p$ defined by the condition $\frac{1}{p}+\frac{1}{p'}=1$. 
In the case of LCA groups the analogue of \eqref{eq:A_p-Rn} is obtained by replacing 
the cubes by base sets. More precisely, 
a weight $w$ is an $A_p=A_p(G,d\mu)$ weight if
\begin{equation}\label{eq:A_p}
[w]_{A_p}:=\sup_{U\in\mathcal{B}} \left(\avgint_U w\,d\mu\right) \left(\avgint_U w^{1-p'}\,d\mu\right)^{p-1} <\infty.
\end{equation}

The limiting case of \eqref{eq:A_p}, when $p=1$, defines the class $A_1$; that is, the set of weights $w$ such that
\begin{equation*}
[w]_{ A_1 }:=\sup_{U\in \mathcal B}\bigg(\avgint_U w\,d\mu \bigg) \esssup_{U} (w^{-1})<+\infty,
\end{equation*}
which is equivalent to $w$ having the property
\begin{equation*}
 Mw(x)\le [w]_{A_1}w(x)\qquad \mu\text{-a.e. } x \in G.
\end{equation*}
As in the usual setting of $\mathbb{R}^n$ we will also often refer to $\sigma:=w^{1-p'}$ as the \emph{dual weight} for $w$. It is easy to verify that $w\in A_p$ if and only if $\sigma\in A_{p'}$. 

The family of $A_p$ classes is increasing and this motivates the definition of the larger class $A_\infty$ as the union $A_\infty=\bigcup_{p\ge 1} A_p$. 
There are many characterizations of the class $A_\infty$  (see \cite{DMRO-Ainfty} or the more classical reference \cite{gra04}). Some of them are given in terms of the finiteness of some $A_\infty$ constant suitably defined. The classical definition consists in taking the limit on the $A_p$ constant as $p$ goes to infinity, namely:
\begin{equation}\label{eq:Ainfty-exp}
(w)_{A_\infty}:=\sup_{U\in\mathcal{B}}\left (\avgint w d\mu \right )\exp{\avgint_U \log(w^{-1}) d\mu}.
\end{equation}
However, the modern tendency is to consider the so-called Fujii-Wilson constant 
implicitly introduced by Fujii in \cite{Fujii}, and later rediscovered by Wilson, \cite{Wilson:87,Wilson-LNM}, and here we choose to follow this approach by defining the $A_\infty$ constant as
\begin{equation}\label{eq:Ainfty-definition}
[w]_{A_\infty}:=\sup_{U\in \mathcal B}\frac{1}{w(U)}\int_U M(w\chi_U)\ d\mu,
\end{equation} 
where $w(U)=\int_Uw\ d\mu$.

\subsection{Our contribution}\label{subsec:results}
As we have already seen, there is a proper -and natural- way to define the $A_p$ and $A_\infty$  classes on an LCA groups having covering families. In contrast with the case $p<\infty$,  it is not immediate that the weight $w$ belongs to $A_\infty$ when any of constants defined on \eqref{eq:Ainfty-exp} and \eqref{eq:Ainfty-definition} is finite. In fact, a weight $w$ is in $A_\infty$ (that is, in some $A_p$) if and only if it satisfies the reverse H\"older inequality which says that 
\begin{equation*}
  \left(\avgint_U w^{r}\ d\mu\right)^{1/r}\leq C \avgint_{\widehat{U}} w\ d\mu
\end{equation*}
for some $r>1$ and where $\widehat{U}$ is an open set defined in terms of $U$ (in the euclidean case $\widehat{U}=U$ and in the case of  spaces of homogeneous type, it is a dilation of $U$).  This is a very well known result in the qualitative case, but it was proved recently in \cite{HPR1} a sharp quantitative result in terms of $[w]_{A_\infty}$ in the context of spaces of homogeneous type. 

Our first result is the following version of the RHI. Note that, as in \cite{HPR1}, we are able to precisely describe the exponent $r$ in term of the constant $[w]_{A_\infty}$.

\begin{theorem}[Sharp  weak reverse H\"older inequality]\label{thm:SharpRHI}

Let $w\in A_\infty$. Define the exponent $r(w)$ as 
\begin{equation*}
r(w)=1+\frac{1}{4D^{10}[w]_{A_\infty}-1},
\end{equation*}
where $D$ is the doubling constant. 
Then, for a fixed $U=x_0+U_{i_0}\in\mathcal{B}$, we have that the following inequality holds
\begin{equation}\label{eq:RHI}
  \left(\avgint_U w^{r(w)}\ d\mu\right)^{1/r(w)}\leq 2D^2 \avgint_{\widehat{U}} w\ d\mu,
\end{equation}
where  $\widehat{U}$ is the  union of the base sets  $\{x+U_i:\, x\in U,\,\,i\leq i_0\}$.
\end{theorem}

Once we have proven such RHI, we are able to provide a quantitative open property for $A_p$ classes. It is very well known that the $A_p$ classes are open in the sense that if $w\in A_p$ for some $p>1$, then $w$ also belongs to some $A_{p-\varepsilon}$ for some $\varepsilon>0$. But the best possible $\varepsilon$ in this property is not completely characterized. Another related interesting and still open question (even in the euclidean setting) is to determine, given a weight $w\in A_\infty$, the smallest $p>1$ such that $w\in A_p$. There are some estimates in \cite{HagPar-Embeddings} but there is no proof of its sharpness.

Here we will deduce from Theorem \ref{thm:SharpRHI} an open property for $A_p$ classes in LCA groups with some control on the constants. More precisely, given $w\in A_p$ for $1<p<\infty$ we will obtain that $w\in A_{p-\varepsilon}$ for $\varepsilon=\frac{p-1}{C[\sigma]_{A_\infty}}$ with $C=4D^{10}$. Further, $[w]_{A_{p-\varepsilon}}\le 2^{p-1}D^{4p-2}[w]_{A_p}$
(se  Lemma \ref{lem:OpenProperty}).

In a recent article \cite{Sauer}  Sauer  proved a weighted bound for the maximal function for LCA groups following Lerner's approach. Additionally, it is asked there if is it possible to obtain the sharp result from Buckley in this general setting. In our main theorem we answer this question by the affirmative and moreover, we provide a better mixed bound. By a mixed bound we understand a bound that depends on $[w]_{A_p}$ and $[w]_{A_\infty}$ of the form $\varphi([w]_{A_p}[w]_{A_\infty})$
where $\varphi$ is some nonegative function, typically a power function. Since we always have that $[w]_{A_\infty}\le [w]_{A_p}$, usually mixed type bounds are sharper than estimates involving only the $A_p$ constant. 

A result in this direction was obtained in  \cite{HPR1} where the authors proved an 
 improvement of Buckley's result \eqref{eq:Buckley} in terms of mixed bounds for spaces of homogeneous type, namely 
\begin{equation*}
\|M\|_{L^p_w\to L^p_w}\le C\left ([w]_{A_p}[\sigma]_{A_\infty}\right )^{\frac{1}{p}}\le C [w]_{A_p}^{\frac{1}{p-1}}.
\end{equation*}

Our main result provides  an extension of the above  estimate to the context  of LCA groups and we will obtain it as a consequence of the RHI and the open property.
We remark here that the lack of geometry in this setting constitutes a major obstacle to overcome.

\begin{theorem}\label{thm:CorSharpBuckley}   Let $M$ be the Hardy-Littlewood maximal function defined in \eqref{eq:Maximal} and let  $1<p<\infty$. Then there is a structural constant $C>0$ such that 
\begin{equation}\label{eq:improvedBuckley}
\|Mf\|_{L_w^p(G)} \le C \left ([w]_{A_p}[\sigma]_{A_\infty}\right )^\frac{1}{p}.
\end{equation}
In particular, 
\begin{equation}\label{eq:Buckley2}
\|M\|_{L^p(w)} \leq C [w]_{A_p}^{\frac{1}{p-1}}.
\end{equation}
\end{theorem}

\subsection{Outline}
The paper is organized as follows. In Section \ref{sec:prelim} we give some preliminary results. We prove the engulfing property in this context that will be used several times along the paper. We also define the local maximal function, prove a crucial covering lemma (Lemma \ref{lem:CZ}) and show a localization property of the local maximal function. In Section \ref{sec:proofs} we give the proofs of the results described in Section \ref{subsec:results}

\section{Preliminaries}\label{sec:prelim}
In this section we provide some properties of  covering families that we will use. Furthermore, we will introduce a local maximal function which  will be crucial to prove the RHI.

As we already mentioned in the introduction the family of dyadic cubes of sidelength $2^i$ or the euclidean balls $B(x,2^i)$ for $i\in \mathbb{Z}$ are covering families for $G=\mathbb{R^d}$. Other examples are presented below.

\begin{example}\noindent

(1) When $G=\mathbb{T}=\{e^{2\pi i t}:\, t\in[-\frac1{2},\frac1{2}) \}$ with the Haar measure  consider $U_k\subseteq G$ defined as $U_0=\mathbb{T}$  and for $k\in\mathbb{N}$, $U_k=\{0\}$ and $U_{-k}=\{e^{2\pi i t}:\, |t|<\frac1{2^{k+1}} \}$. Then, $\{U_k\}_{k\in\mathbb{Z}}$ is a covering family for $\mathbb{T}$ with $\theta(k)=k+1$ and $D=2$.

(2) For $G=\mathbb{Z}$ take $U_i=\{k\in\mathbb{Z}:\, |k|\leq 2^{i-1}\}$ for $i\geq 1$ and $U_i=\{0\}$ otherwise. Then $\{U_i\}_{i\in\mathbb{Z}}$ is a covering family for $\mathbb{Z}$  with $\theta(i)=i+1$ and $D=2$.

 (3) Let $G$ be an LCA group with Haar measure $\mu$ and let $H$ be  a compact and open subgroup of $G$ with  $\mu(H)=1$. Consider an expansive automorphism  $A:G\to G$ with respect to $H$, which means that $H\subsetneq AH$ and $\bigcap_{i<0}A^iH=\{0\}$. If additionally, $G=\bigcup_{i\in\mathbb{Z}}A^iH$, then $\{A^iH\}_{i\in\mathbb{Z}}$ is a covering family for $G$. Indeed. Since $H\subsetneq AH$ and $H$ is a group, $A^iH-A^iH=A^iH\subseteq A^{i+1}H$ so $\theta(i)=i+1$. To see that the doubling property is satisfied, note that $\mu_A$ defined as $\mu_A(B):=\mu(AB)$ for $B\subseteq G$ a Borel set, is a Haar measure on $G$. Thus, there is a positive number $\alpha$ such that $\mu_A=\alpha\mu$. The constant $\alpha$ is the so-called {\it modulus of $A$} and is denoted by $\alpha=|A|$. Then, $\mu(A^{i+1}H)=\mu_A(A^{i}H)=|A|\mu(A^{i}H)$ for $i\in\mathbb{Z}$. Observe that  $G/H$ is discrete and $AH/H$ is finite, so $AH$ is the union of finitely many  cosets of the quotient $G/H$, say $\{H+s_j\}_{j=1}^r$. Therefore, $|A|=|A|\mu(H)=\mu(AH)=r$ and since $H\subsetneq AH$, $r\geq2$. Thus we can take $D=|A|\geq 2$. A structure  of this type is considered in \cite{BB04} for constructing wavelets on LCA groups with open and compact subgroups.\\
For a concrete example of this situation, consider the $p$-adic group  $G=\mathbb{Q}_p$
where $p\geq 2$ is a prime number consisting  of all formal Laurent series in $p$ with coefficients $\{0, 1, . . . , p-1\}$, that is, 
$$\mathbb{Q}_p=\left\{\sum_{n\geq n_0} a_np^n:\, n_0\in\mathbb{Z},\,\, a_n\in \{0, 1, . . . , p-1\}\right\}.$$
As a  compact and open subgroup we can consider $H=\mathbb{Z}_p$ which is 
$$\mathbb{Z}_p=\left\{\sum_{n\geq 0} a_np^n:\,  a_n\in \{0, 1, . . . , p-1\}\right\}.$$
Take $A:\mathbb{Q}_p\to \mathbb{Q}_p$ the automorphism defined as $A(x)=p^{-1}x$. Then, $A$ is expansive with respect to $\mathbb{Z}_p$ and it can be easily checked that $\mathbb{Q}_p=\bigcup_{i\in\mathbb{Z}}A^i\mathbb{Z}_p$. Then, $\{A^i\mathbb{Z}_p\}_{i\in\mathbb{Z}}$ is a covering family for $\mathbb{Q}_p$ and in  this case, $D=|A|=p$.
\end{example}

Let $\{U_i\}_{i\in\mathbb{Z}}$ be a fixed covering family for $G$.
From now on, we assume the sets $U_i$ to be  symmetric. This is not a restriction at all because one can always consider the new family of base sets formed by the difference sets $U_i-U_i$ which  increases  the doubling constant from $D$ to $D^2$. We denote $2U_i:=U_i-U_i=U_i+U_i$. 

  Any covering family has the so-called \emph{engulfing} property:
\begin{lemma}\label{lem:engulfing-property}
Let $U,V$ be two base sets such that $U=x+U_i$ and $V=y+U_j$ with $i\le j$ and $x,y\in G$. If $U\cap V\neq\emptyset$, then $x+U_i\subset y+U_{\theta^2(j)}$.
\end{lemma}

\begin{proof}
There are two point $u_i\in U_i$ and $u_j\in U_j$ such that $x+u_i=y+u_j$. Then $x=y+u_i-u_j\in y+ U_j-U_j\subset y+U_{\theta(j)}$ and therefore $x+U_i\subset y+U_{\theta(j)}+U_{\theta(j)}\subset y+U_{\theta^2(j)}$ (recall that we assume that the base sets are symmetric).
\end{proof}

\begin{remark}\label{rem:jota}
For a given $V\in \mathcal{B}$, where $\mathcal{B}$ is the base of $G$ defined as in \eqref{eq:basesets}
we will denote by $j(V)\in\mathbb{Z}$  the maximum integer such that  $V=x+U_{j(V)}$ for some $x\in G$. To see that such a number exists, let us define $N(V)=\{j\in \mathbb{Z}: \exists \,x\in G , V=x+U_j\}$ and show that  $\sup N(V)<\infty$. If $\sup N(V)=\infty$, we could find a sequence $\{x_n\}_{n\in \mathbb{N}}$ of points in $G$ and a sequence of integer indices $\{i_n\}_{n\in \mathbb{N}}$ such that $i_n\to\infty$ as $n\to\infty$ and 
$$
V=x_n+U_{i_n} \qquad \text{ for all } n\in\mathbb{N}.
$$

By compactness of $\overline{V}$ we can assume (relabelling) that the sequence converges to some $x\in G$, which we can assume to be the origin. Now we claim that, for any $j\in \mathbb{N}$, there is some $m\in \mathbb{N}$ such that $U_j\subset x_m+U_{i_m}$ and from this fact would follow that $\mu(V)=\infty$, but this implies that $\infty =\mu(V)\le \mu(\overline{V})<\infty$ whis is a contradiction.   To verify the claim, fix $U_j$ and choose $n_0$ such that $x_n\in U_j$  and $i_n\geq j$ for all $n\ge n_0$. Then we have that
$$
U_j\cap x_n + U_{i_n}\neq \emptyset
$$
for all $n\ge n_0$. Furthermore, the above still holds if we replace $x_n$  by any $x_m$ with $m\ge n \ge n_0$ since $x_m\in U_j$ and $x_m\in x_m + U_{i_n}$. Therefore by the engulfing property (see e.g. Lemma \ref{lem:engulfing-property}) we obtain that 
$$
U_j\subset x_m + U_{\theta^2(i_n)} \subset x_m + U_{i_m}
$$
for any $m$ such that $i_m\ge \theta^2(i_n)$.
\end{remark}

In order to introduce the local maximal function, we first define a local base for a fixed base set $U$.  

\begin{definition}\label{def:localBase}
Let $U\in\mathcal{B}$ be a fixed base set and let  $k:=j(U)$. The local base $\mathcal{B}_U$ is defined as
\begin{equation}\label{eq:localBase}
\mathcal{B}_U:=\left \{y+U_j: y\in U, j\le k\right \}.
\end{equation}
We also defined the {\it enlarged} set $\widehat{U}$ by the formula
\begin{equation}\label{eq:U-hat}
\widehat{U}:=\bigcup_{V\in \mathcal{B}_U}V.
\end{equation}
\end{definition}

\begin{lemma}\label{lem:engulfing}
Let $U=x+U_{k}$ be a fixed base set in $\mathcal{B}$ and set $k=j(U)$. 
We then have the following geometric properties:
\begin{enumerate}
\item[$(a)$] If $V\in  \mathcal{B}_U$ then $V\subset x+U_{\theta(k)}$. 
\item[$(b)$] For any $z\in U$, we have that 
\begin{equation*}\label{eq:engulfing1}
\widehat{U}\subset z+U_{\theta^2(k)},
\end{equation*}
where $\widehat{U}$ is as in \eqref{eq:U-hat}.
As a consequence of this last property, we obtain 
\begin{equation*}\label{eq:engulfing2}
\mu(\widehat{U})\le \mu(z+U_{\theta^2(k)})\le D^2\mu(z+U_k)
\end{equation*}
for any $z\in U$. In particular, since $U=x+U_k$, $\mu(\widehat{U})\le D^2\mu(U)$.
\end{enumerate}
\end{lemma}
\begin{proof}
$(a)$. Let $V=y+U_j$ with $j\le k$ and take any $z\in V$. Then $z=y+u_j$ with $u_j\in U_j\subset U_k$. Since $y\in U$ we can write $y=x+u_k$, $u_k\in U_k$. Then we have that $z=x+u_j+u_k\in x+U_k+U_k\subset x+U_{\theta(k)}$.

$(b)$. Let $V\in \mathcal{B}_U$, $V=y+U_j$ with $y\in U$, $j\le k$. By Lemma \ref{lem:engulfing-property}, since $V\cap U\neq \emptyset$, we have that $V\subset x+U_{\theta(k)}$. Take any $z\in U$, $z=x+u_k$, $u_k\in U_k$. Then
\begin{equation*}
V\subset x+U_{\theta(k)}= z-u_k+U_{\theta(k)}\subset z-U_k+U_{\theta(k)}\subset z-U_{\theta(k)}+U_{\theta(k)}\subset z+U_{\theta^2(k)}.
\end{equation*}
\end{proof}

We now define the local maximal function as follows
\begin{equation}\label{eq:localMaximal}
M_Uf(y):=\sup_{y\in V\in\mathcal{B}_U}\avgint_V |f(z)| \ d\mu(z)
\end{equation}
for any $y\in\widehat{U}$  and and $M_Uf(y)=0$ otherwise.

\begin{remark}\label{rem:lebesgue} $(a)$
In \cite[Theorem 44.18]{HR-vol2}, it is proven a version of the Lebesgue Differentiation Theorem with respect to the Haar measure for LCA groups having a $D'$-sequences (cf. \cite[Definition 44.10]{HR-vol2}). A careful reading of the proof of  \cite[Theorem 44.18]{HR-vol2} reveals that the result is still true with the obvious changes  for measures which are not translation invariant. Thus, since a covering family is in particular a $D'$-sequence, we have that the Lebesgue Differentiation Theorem holds in our context. 

$(b)$
As a consequence of the Lebesgue Differentiation Theorem, we have the elementary but important property of the local  maximal function:  \\$f(x)\le M_Uf(x)$ $\mu$-almost everywhere $x\in U$.
\end{remark}

Consider now, for a fixed $U\in\mathcal{B}$,  the level set for the local maximal function acting on a weight $w$ at scale $\lambda>0$:
\begin{equation}\label{eq:Omegalambda}
\Omega_\lambda:=\left \{x\in \widehat{U}: M_Uw(x)>\lambda\right \}.
\end{equation}

A key instrument will be a Calder\'on-Zygmund decomposition of $\Omega_\lambda$. We will obtain it by using an adapted version of a covering lemma from \cite[Lemma 2.2.1]{EG77}.  Although the proof follows standard arguments, we include it here for completeness.
When $w$ be a nonnegative and locally integrable function on $G$ and $V\subseteq G$ is relatively compact we denote the average of $w$ on $V$ as $w_V$; that is, $w_V=\avgint_{V}w\ d\mu$.

\begin{lemma}\label{lem:CZ}
Let $U\in\mathcal{B}$ be a fixed base set  in $G$ and let $w$ be a nonnegative and integrable function supported on $\widehat{U}$. For  $\lambda>w_{\widehat{U}}$, define $\Omega_\lambda$ as in \eqref{eq:Omegalambda}. If $\Omega_\lambda$ is nonempty, there exists a finite or countable index set $Q$ and a family $\{y_i+U_{{\alpha_i}}\}_{i\in Q}$ of pairwise disjoint base sets from $\mathcal{B}_U$ such that
\begin{enumerate}
\item[$(a)$]The sequence $\{{\alpha_i}\}_{i\in Q}$ is decreasing.
\item[$(b)$] $\displaystyle\bigcup_{i\in Q} y_i+U_{{\alpha_i}}\subset\Omega_\lambda\subset \displaystyle\bigcup_{i\in Q} y_i+U_{\theta^2(\alpha_i)}$.
\item[$(c)$] For any $i\in Q$, we have that 
\begin{equation*}
\lambda < \avgint_{y_i+U_{{\alpha_i}}} w\ d\mu.
\end{equation*} 
\item[$(d)$] Given $r>{\alpha_i}$ for some $i\in Q$, then 
\begin{equation}\label{eq:big-average}
\avgint_{y_i+U_r}w\ d\mu \le D^2\lambda.
\end{equation}
\end{enumerate}
\end{lemma}

\begin{proof}
Suppose that there is no finite sequence of points in $\Omega_\lambda$ such that the conclusion holds (in that case, there is nothing to prove). 
For $x\in \Omega_\lambda$, define 
\begin{equation}
\alpha(x)=\max \left \{j\in\mathbb{Z}: \exists\, V=y+U_j \in \mathcal{B}_U, x\in V, \avgint_{V}w\ d\mu>\lambda\right \}.
\end{equation}
Since $V=y+U_j \in \mathcal{B}_U$ implies  $j\leq j(U)$, we have that $\alpha$ is well-defined.
Consider now, for each $x\in \Omega_\lambda$ a base set $V_x\in   \mathcal{B}_U$,  $V_x:=y_x+U_{\alpha(x)}$ such that $x\in V_x$. In other words, one of the base sets in $\mathcal{B}$ containing the point $x$ where the map $\alpha$ attains its value. Observe that in particular, $\alpha(y_x)\geq\alpha(x)$.

We start by picking  $x_1$ as a extremal point for $\alpha$, that is $\alpha(x_1)\ge \alpha(x)$ for all $x\in\Omega_\lambda$. Put $\alpha_1=\alpha(x_1)$ and $y_1:=y_{x_1}$ such that $V_{x_1}=y_1+U_{\alpha_1}$. Note that, since $\alpha_1\le\alpha(y_1)\le\alpha(x_1)=\alpha_1$, we also have that $\alpha(y_1)=\alpha_1$.
Now suppose that we have chosen the first $n$ points $y_1,\dots, y_n$ and their respective base sets $U_{\alpha_1},\dots,U_{\alpha_n}$ such that
\begin{itemize}
\item the sets $V_j:= y_j+U_{\alpha_j}$, $1\le j\le n$ are pairwise disjoint,
\item $\alpha_j:=\alpha(y_j)\ge \alpha(x)$ for all $x\in A_{j-1}$, where 
\begin{equation}
A_j:=\Omega_\lambda\setminus \bigcup_{\ell\le j} y_\ell +U_{\theta^2(\alpha_\ell)}\qquad 1\le j \le n.
\end{equation}
\end{itemize}
Since we are assuming that this procedure never ends, we have that $A_j\neq\emptyset$ for all $1\leq j\leq n$. 
Therefore we can choose $ x_{n+1}\in A_n$ such that $\alpha_{n+1}:=\alpha(x_{n+1})\geq \alpha(x)$ for all $x\in A_n$. This means that there is base set $V_{n+1}:=y_{n+1}+U_{\alpha_{n+1}}$ and in particular $w_{V_{n+1}}>\lambda$ and $\alpha(y_{n+1})=\alpha_{n+1}$. Note that this construction produces a decreasing sequence $\{\alpha_n\}_{n\in \mathbb{N}}$. Let's see that $V_{n+1}\cap V_j=\emptyset$ for all $1\le j \le n$. Supposing that this is not the case, we could find $u\in U_{\alpha_{n+1}}$ and $v\in U_{\alpha_j}$ for some $j\le n$ such that 
\begin{equation*}
y_{n+1}+u=y_j+v.
\end{equation*}
Since $x_{n+1}\in V_{n+1}$, we have that for some $z\in U_{\alpha_{n+1}}$,
\begin{equation*}
x_{n+1}=y_{n+1}+z=y_j+v-u+z\in y_j+ U_{\alpha_j}-U_{\alpha_{n+1}}+U_{\alpha_{n+1}}.
\end{equation*}
Since $U_{\alpha_{n+1}}\subset U_{\alpha_{j}}$ and trivially $U_{\alpha_{j}}\subset U_{\theta(\alpha_{j})}$, we get that 
\begin{equation*}
 x_{n+1}\in y_j+U_{\theta^2(\alpha_{j})},
 \end{equation*} 
which is a contradiction by the choice of $x_{n+1}$.

We are left to prove that this procedure exhausts the set $\Omega_\lambda$. If not, there is a point $x\in A_n$  with $\alpha(x)\le \alpha_n$ for all $n\ge 1$. Define the set  $S$ as
\begin{equation*}
S:=\{y_n: n\in\mathbb{N}\}.
\end{equation*}
Since 
\begin{equation*}
S\subset \{z\in \Omega_\lambda: \alpha(z)\ge\alpha(x)\}\subset \widehat{U}
\end{equation*}
and $\widehat{U}$ is contained in some base set (see item (b) in Lemma \ref{lem:engulfing})
we conclude that $S$ is relatively compact.

By monotonicity of $\alpha$, we have that $U_{\alpha_n}\subset U_{\alpha_1}$. Therefore the set 
\begin{equation*}
F:=\bigcup_n (y_n+U_{\alpha_n})\subset S+U_{\alpha_1}
\end{equation*}
is also relatively compact and this implies that $\mu(\overline{F})<\infty$. Now consider $N\in \mathbb{Z}$ such that $\overline{S}\subset U_N$ and an integer $r>0$ such that $\theta^r(\alpha(x))\ge N$. Then we have that for any $n\in\mathbb{N}$, $y_n\in S\subset U_N\subset U_{\theta^r(\alpha(x))}$ and thus   $0\in y_n+U_{\theta^r(\alpha(x))}$. Further, we obtain that  
\begin{equation*}
U_N=0+U_N\subset y_n+U_{\theta^r(\alpha(x))}+U_N\subset y_n+2U_{\theta^r(\alpha(x))}\subset y_n+U_{\theta^{r+1}(\alpha(x))}.
\end{equation*}
The doubling property shows that 
\begin{equation*}
\mu(U_N)\le D^{r+1}\mu(y_n +U_{\alpha(x)})
\end{equation*}
and this implies that
\begin{equation*}
\mu(F)=\sum_n\mu(y_n+U_{\alpha_n})\ge \sum_n\mu(y_n+U_{\alpha(x)})
\ge  D^{-(r+1)}\sum_n\mu(U_N)=\infty.
\end{equation*}
This contradicts the condition $\mu(\overline{F})<\infty$  and we conclude with the proof of items (a), (b) and (c) of the lemma.

We prove now item (d). Towards to control the average on $y_i+ U_r$ we consider two cases: 
first we consider $r\leq k:=j(U)$. Then $y_i+ U_{r}\in \mathcal{B}_U$ and by maximality we have 
$\displaystyle \avgint_{y_i+ U_r} w \ d\mu\leq\lambda$. Indeed, if not we would have that $\alpha_i=\alpha(y_i)\geq r > \alpha_i$. 
Second, in case 
$r> k$, we have that $\theta^2(r)> \theta^2(k)$ and thus, by Lemma \ref{lem:engulfing},
$y_i+U_{\theta^2(r)}\supset 
y_i+U_{\theta^2(k)}\supset\widehat{U}$. Therefore, since $w=0$ a.e $\widehat{U}^c$, we have 
$$
\avgint_{y_i+ U_r }w \ d\mu\leq \frac{\mu(\widehat{U})}{\mu(y_i+ U_r)}\avgint_{\widehat{U}}w \ d\mu\leq D^2\lambda. 
$$
The lemma is now completely proven.
\end{proof}

Now we present a localization argument for the local maximal function $M_U$. The idea is better understood when considering the usual dyadic maximal function $M^d_Q$ localized on a cube $Q$ in $\mathbb{R}^n$. Suppose that the level set $\Omega_\lambda=\{x\in Q: M^d_Qw(x)>\lambda\}$ for $\lambda>w_Q$ is decomposed into dyadic subcubes of $Q$ such that $Q=\bigcup_i Q_i$ and the cubes $Q_i$ are maximal with respect to the condition $w_{Q_i}>\lambda$. Then the conclusion is that for any $x\in Q_i$ the equality $M^d_Qw(x)=M^d_Q(w\chi_{Q_i})(x)$ holds. In this more general setting, the analogous result is contained in the following lemma which has not a direct proof as in the dyadic case. 

For simplicity in the exposition, we introduce the following notation. Given a base set of the form $V=y+U_j$ we denote by $V^*$ the dilation of $V$ by $\theta$, i.e. $V^*=y+U_{\theta(j)}$. Further iterations of this operation are defined recursively, that is, $V^{**}=(V^*)^*$ and $V^{n*}$ for $n$ iterations of the dilation operation.

\begin{lemma}\label{lem:local-max}
Let $U\in\mathcal{B}$ be a fixed base set and consider  $w=w\chi_{\widehat{U}}$ a nonnegative and integrable function on $\widehat{U}$ where $\widehat{U}$  is as in \eqref{eq:U-hat}. For   $\lambda>w_{\widehat{U}}$ let $\Omega_\lambda$ defined as above and let $\{V_i\}_{i\in Q}=\{y_i+U_{\alpha_i}\}_{i\in Q}$ be the C-Z decomposition of $\Omega_\lambda$ given by Lemma \ref{lem:CZ}. Then, for $L=D^6$ any $i\in Q$ and any $x\in V_i^{**}\cap \Omega_{L\lambda}$ we have
\begin{equation}\label{eq:local-max}
M_Uw(x)\le M_U(w\chi_{V_i^{4*}})(x).
\end{equation}
\end{lemma}
\begin{proof}
Let $x\in V^{**}_i\cap\Omega_{L\lambda}$. Then there exists $V\in\mathcal{B}_U$, $V=y+U_j$, with $y\in U$ and $j\le j(U)$ such that $x\in V$ and $w_V>L\lambda$. We claim that $j\le \theta^2(\alpha_i)$. To see that this is in fact true, suppose towards a contradiction, that $j>\theta^2(\alpha_i)$. Then, $V\subset y_i+U_{\theta^2(j)}$. Indeed, if $z\in V$ then $z=y+w$ with $w\in U_j$. On the other hand, since $x\in V^{**}_i\cap V$, $x=y_i+u=y+v$ with $u\in U_{\theta^2(\alpha_i)}$ and $v\in U_j$. Then 
\begin{equation*}
z=y+v-v+w=x -v+w=y_i+ u-v+w.
\end{equation*}
Since  $U_{\theta^2(\alpha_i)} \subset U_j$, we obtain that $z\in y_i+ U_j+U_{\theta(j)}\subset y_i+ U_{\theta^2(j)}$. As a consequence, 
\begin{equation*}
\avgint_V w\ d\mu\le \frac{\mu(y_i+ U_{\theta^2(j)})}{\mu(V)}\avgint_{y_i+ U_{\theta^2(j)}}w \ d\mu.
\end{equation*}
We note that since $\theta^2(\alpha_i)<j$, $x\in V\cap V^{**}_i\subset V\cap(y_i+U_j)$ and then, by the engulfing property we have that $y_i+U_j\subset y+U_{\theta^2(j)}$. 
Thus, using the doubling property of the measure $\mu$ we obtain 
$$
\frac{\mu(y_i+ U_{\theta^2(j)})}{\mu(y+U_j)}\leq D^2 \frac{\mu(y_i+ U_j)}{\mu(y+U_j)}
\leq D^2\frac{\mu(y+ U_{\theta^2(j)})}{\mu(y+U_j)}\leq D^4.
$$

Furthermore, since $\theta^2(j)\geq j>\theta^2(\alpha_i)\geq\alpha_i$, by item (4) in Lemma \eqref{lem:CZ} we have that
$$
\avgint_{y_i+ U_{\theta^2(j)}}w \ d\mu\leq 
D^2\lambda 
$$
and we can  conclude that 
\begin{equation*}
L\lambda<\avgint_V w\ d\mu\le D^6\lambda=L\lambda,
\end{equation*}
which gives a contradiction. Hence, the claim $j\le \theta^2(\alpha_i)$ holds. 

Now, using Lemma \ref{lem:engulfing-property} we have that $V\subset V^{4*}_i$ and then 
$$
\avgint_Vw\ d\mu=\avgint_Vw\chi_{V^{4*}_i}\ d\mu\le M(w\chi_{V^{4*}_i})(x)
$$
which proves inequality \eqref{eq:local-max}.

\end{proof}

\section{Proof of the main results}\label{sec:proofs}

We present here the proof of Theorem \ref{thm:SharpRHI}.

\begin{proof}[Proof of Theorem \ref{thm:SharpRHI}]

Step 1. We start with the following estimate for the local maximal function. Let $U=x_0+U_k$ be a fixed base set. We claim that, for $\varepsilon=\frac{1}{4D^{10}[w]_{A_\infty}-1}$, we have that
\begin{equation}\label{eq:Maximal-vs-w}
\avgint_{\widehat U} (M_Uw)^{1+\varepsilon} \ d\mu \le 2 [w]_{A_\infty}\left (\avgint_{\widehat U} w \ d\mu\right )^{1+\varepsilon}.
\end{equation}
Recall that we may assume that the weight $w$ is supported on $\widehat U$. Let $\Omega_\lambda$ defined as in \eqref{eq:Omegalambda}. We write the norm using the layer cake formula as follows
\begin{eqnarray*}
\int_{\widehat U} (M_Uw)^{1+\varepsilon} \ d\mu & = & \int_0^{\infty} \varepsilon \lambda^{\varepsilon-1}M_Uw(\Omega_\lambda)\ d\lambda\\
& = & \int_0^{w_{\widehat U}} \varepsilon \lambda^{\varepsilon-1}M_Uw(\Omega_\lambda)\ d\lambda
+ \int_{w_{\widehat U}}^\infty \varepsilon \lambda^{\varepsilon-1}M_Uw(\Omega_\lambda)\ d\lambda\\
& = & I + II.
\end{eqnarray*}

The first term is easily controlled by using the $A_\infty$ constant of $w$ (see \eqref{eq:Ainfty-definition}):
\begin{eqnarray*}
I\le M_Uw(\widehat U)w_{\widehat U}^\varepsilon & = &  w_{\widehat U}^\varepsilon\int_{\widehat U}M_Uw \ d\mu \\
 & \le & w_{\widehat U}^\varepsilon\int_{y+U_{\theta^2(k)}}M_U(w\chi_{y+U_{\theta^2(k)}}) \ d\mu\\
& \le &  w_{\widehat U}^\varepsilon [w]_{A_\infty}w(y+U_{\theta^2(k)})\\
& = &  w_{\widehat U}^\varepsilon [w]_{A_\infty}w(\widehat{U})
\end{eqnarray*}
where $y\in U$ and we used Lemma \ref{lem:engulfing} and the definition of $[w]_{A_\infty}$.

Now, for each $\lambda>w_{\widehat U}$ we consider $\{V_i\}_{i\in Q}$ the C-Z decomposition of $\Omega_\lambda$ from Lemma \ref{lem:CZ} to control $II$. We have that
\begin{eqnarray*}
M_Uw(\Omega_\lambda)\le \sum_i M_Uw(V_i^{**}).
\end{eqnarray*}
For any $i\in Q$ we write $V^{**}_i=V_1\cup V_2$ with 
$V_1:=V^{**}_i\cap\Omega_{L\lambda}$ 
and $V_2:=V^{**}_i\setminus\Omega_{L\lambda}$ where $L=D^6$. Thus, by Lemma \ref{lem:local-max} and the $A_{\infty}$ property \eqref{eq:Ainfty-definition} we have 
\begin{eqnarray*}
M_Uw(V_i^{**}) &=& \int_{V_1}M_Uw \ d\mu + \int_{V_2}M_Uw \ d\mu\\
& \le & \int_{V_1}M_U(w\chi_{V^{4*}_i})(x) \ d\mu + L\lambda\mu(V_2)\\
& \le & [w]_{A_{\infty}} w(V^{4*}_i) + L\lambda\mu(V^{4*}_i) =  \left([w]_{A_{\infty}} w_{V^{4*}_i}+ L\lambda\right)\mu(V^{4*}_i)\\\\
& \le & \left([w]_{A_{\infty}} \lambda D^2+ L\lambda\right)D^4\mu(V_i)\le 2[w]_{A_{\infty}} \lambda D^{10} \mu(V_i),
\end{eqnarray*}
where in the last inequality we have used \eqref{eq:big-average} and the doubling property of $\mu$. This gives
\begin{eqnarray*}
M_Uw(\Omega_\lambda) & \le & \sum_i M_Uw(V_i^{**})\le  2[w]_{A_{\infty}} \lambda D^{10}\sum_i \mu(V_i)\\
& \le & 2[w]_{A_{\infty}} \lambda D^{10}\mu(\Omega_\lambda).
\end{eqnarray*}
Thus, 
\begin{eqnarray*}
II & \le & 2[w]_{A_{\infty}} D^{10} \int_0^\infty \varepsilon \lambda^\varepsilon\mu(\Omega_\lambda)\ d\lambda\\
& = &  2[w]_{A_{\infty}} D^{10}\frac{\varepsilon}{\varepsilon+1}\int_{\widehat U}M_Uw^{1+\varepsilon}\, d\mu.
\end{eqnarray*}
Therefore, gathering all the estimations and averaging over $\widehat U$, we have that 
$$
\left(1- 2[w]_{A_{\infty}} D^{10} \frac{\varepsilon}{\varepsilon+1}\right)\avgint_{\widehat U}M_Uw^{1+\varepsilon}\, d\mu \le w_{\widehat U}^{1+\varepsilon}.
$$
Choosing $\varepsilon\le \frac{1}{4[w]_{A_{\infty}} D^{10}-1}$ we get that $1- 2[w]_{A_{\infty}} D^{10} \frac{\varepsilon}{\varepsilon+1}\ge \frac{1}{2}$ and we obtain the desired estimate \eqref{eq:Maximal-vs-w}.

Step 2. Now we proceed to prove the main estimate \eqref{eq:RHI}. By Remark \ref{rem:lebesgue} we have that  $w(x)\le M_Uw(x)$  holds on $U$. Then we obtain
\begin{equation*}
\int_U w^{1+\varepsilon}\ d\mu\le \int_{U} (M_Uw)^\varepsilon w \ d\mu\le \int_{\widehat{U}} (M_Uw)^\varepsilon w \ d\mu.
\end{equation*}
Once again we use the layer cake formula combined with the C-Z decomposition of $\Omega_\lambda$ and proceeding in a similar way as above we obtain
\begin{eqnarray*}
\int_{\widehat U} (M_Uw)^{\varepsilon} w \ d\mu & = & \int_0^{\infty} \varepsilon \lambda^{\varepsilon-1}w(\Omega_\lambda)\ d\lambda\\
& = & \int_0^{w_{\widehat U}} \varepsilon \lambda^{\varepsilon-1}w(\Omega_\lambda)\ d\lambda
+ \int_{w_{\widehat U}}^\infty \varepsilon \lambda^{\varepsilon-1}w(\Omega_\lambda)\ d\lambda\\
& \le & w(\widehat U)w_{\widehat U}^\varepsilon + \int_{w_{\widehat U}}^\infty \varepsilon \lambda^{\varepsilon-1}\sum_i w(V_i^{**})\ d\lambda\\
& \le & w(\widehat U)w_{\widehat U}^\varepsilon + D^2\int_{w_{\widehat U}}^\infty \varepsilon \lambda^{\varepsilon}\sum_i \mu(V_i^{**})\ d\lambda\\
& \le & w(\widehat U)w_{\widehat U}^\varepsilon + D^4\int_{w_{\widehat U}}^\infty \varepsilon \lambda^{\varepsilon}\sum_i \mu(V_i)\ d\lambda\\
& \le & w(\widehat U)w_{\widehat U}^\varepsilon + D^4\int_{0}^{\infty}\varepsilon \lambda^{\varepsilon} \mu(\Omega_\lambda)\ d\lambda\\
& \le & w(\widehat U)w_{\widehat U}^\varepsilon + \frac{D^4\varepsilon}{\varepsilon+1}\int_{\widehat U}(M_Uw)^{1+\varepsilon}\ d\mu\\
\end{eqnarray*}
Therefore, averaging over $U$, using that $\mu(\widehat{U})\le D^2 \mu(U)$ and \eqref{eq:Maximal-vs-w} we have
\begin{equation*}
\avgint_U w^{1+\varepsilon}\ d\mu\le D^2 w_{\widehat U}^{\varepsilon+1} + \frac{2D^6\varepsilon[w]_{A_\infty}}{\varepsilon+1}\left(\avgint_{\widehat U}w\ d\mu\right)^{1+\varepsilon}.
\end{equation*}
By our previous choice of $\varepsilon$, $\frac{2D^6\varepsilon[w]_{A_\infty}}{\varepsilon+1}\le \frac{2D^{10}\varepsilon[w]_{A_\infty}}{\varepsilon+1}\le \frac{1}{2}$
and we conclude that 
\begin{equation*}
\avgint_U w^{1+\varepsilon}\ d\mu \le 2D^2 \left(\avgint_{\widehat U}w\ d\mu\right)^{1+\varepsilon}.
\end{equation*}
\end{proof}

We present now some classical applications of the RHI to weighted norm inequalities for maximal functions. One crucial property of $A_p$ classes is the well known open condition. In the  next lemma we provide a quantitative version of it. 

\begin{lemma}\label{lem:OpenProperty}
For $1<p<\infty$, let $w\in A_p$ . Then, for $\varepsilon=\frac{p-1}{C[\sigma]_{A_\infty}}$ with $C=4D^{10}$ and $\sigma=w^{1-p'}$, we have that $w\in A_{p-\varepsilon}$. Further, $$[w]_{A_{p-\varepsilon}}\le 2^{p-1}D^{4p-2}[w]_{A_p}.$$
\end{lemma}
\begin{proof}
Let $w\in A_p$. The $A_{p-\varepsilon}$ condition for $w$ takes the form
\begin{equation*}
\sup_{U\in \mathcal{B}}\left (\avgint_U w\ d\mu\right )\left (\avgint_U w^{1-{(p-\varepsilon)'}}\ d\mu\right )^{p-\varepsilon-1}<\infty.
\end{equation*}
Recall that the dual weight of $w$, $\sigma=w^{1-p'}$  is also in $A_\infty$. Therefore it satisfies a RHI with exponent $r(\sigma)$ given by Theorem \ref{thm:SharpRHI}. Choose $\varepsilon$ such that $1-(p-\varepsilon)'=(1-p')r(\sigma)	$, namely $\varepsilon=\frac{p-1}{r(\sigma)'}$ which is equivalent to the condition $r(\sigma)=\frac{p-1}{p-\varepsilon-1}$. Then we obtain
\begin{eqnarray*}
\left (\avgint_U w^{1-(p-\varepsilon)'\ d\mu}\right )^{p-\varepsilon-1} & = & \left (\avgint_U \sigma^{(1-p')r(\sigma)}\ d\mu\right )^{\frac{p-1}{r(\sigma)}}\\
&\le & \left (2D^2 \avgint_{\widehat{U}} \sigma\ d\mu\right)^{p-1},
\end{eqnarray*}
for any $U\in \mathcal{B}$. Now, for $U=x+U_k \in \mathcal{B}$, recall that $U^{**}=x+U_{\theta^2(k)}$ and that $\widehat{U}\subset U^{**}$. Then we have that
\begin{eqnarray*}
\left (\avgint_U w\ d\mu\right )\left (\avgint_U w^{1-{(p-\varepsilon)'}}\ d\mu\right )^{p-\varepsilon-1}
& \le &  C \left (\avgint_{U^{**}} w\ d\mu\right )\left (\avgint_{U^{**}} \sigma \ d\mu\right )^{p-1}
\end{eqnarray*}
with $C=2^{p-1}D^{4p-2}$. We conclude that 
\begin{equation*}
[w]_{A_{p-\varepsilon}}\le 2^{p-1}D^{4p-2} [w]_{A_p}.
\end{equation*}
\end{proof}

In what follows we will need the fact that the maximal function $M$ maps $L^{q,\infty}_w(G)$ to  itself with operator norm bounded by $C[w]^{\frac{1}{q}}_{A_q}$ for some $C>0$. Without presenting any details on weak norms and Lorentz spaces, we include here a quantitative estimate on the size of level sets of the maximal function.

\begin{lemma}\label{lem:weak}
Let $1\le q<\infty$ and let $M$ the maximal function defined in \eqref{eq:Maximal}. Then, for any $f\in L_w^q(G)$, we have that  
\begin{equation}\label{eq:weakM}
\sup_{\lambda>0}\lambda^q w(\{x\in G: Mf(x)>\lambda\})\le D^{2q} [w]_{A_q}\|f\|^q_{L^q_w}.
\end{equation}
\end{lemma}

\begin{proof}
For any locally integrable function $f$ and any $\lambda>0$ let $\Omega_\lambda$ be  the level set $\Omega_\lambda:=\{x\in G: Mf(x)>\lambda\}$. We also define some sort of \emph{truncated} maximal operator as follows:  for any $K\in \mathbb{Z}$, let $M_K$ the averaging operator given by 
\begin{equation}\label{eq:truncated-maximal}
M_Kf(x)=\sup_{V\in\mathcal{B}_K(x)}\avgint_V |f(z)|  d\mu,
\end{equation}
where the supremum is taken over the subfamily $\mathcal{B}_K$ of $\mathcal{B}$ consisting of all base sets of the form $y+U_i$ with $y\in G$ and $i\le K$ containing the point $x$, i.e.:
\begin{equation}\label{eq:BK(x)}
 \mathcal{B}_K(x):=\{V=y+U_i: x\in V, i\le K\}.
 \end{equation} 
For each $K$ we consider the corresponding level set $\Omega^K_\lambda:=\{x\in G: M_Kf(x)>\lambda\}$. We clearly have that the family $\{\Omega^K_\lambda\}$ is increasing in $K$ and also $\Omega_\lambda=\bigcup_K\Omega^K_\lambda$. We therefore may compute the value of $w(\Omega_\lambda)$ as the limit of $w(\Omega^K_\lambda)$. In addition, we recall that the group $G$ is $\sigma$-compact since $G=\bigcup_{r\in \mathbb{Z}} \overline{U_r}$. We will use again a limiting argument to compute $w(\Omega^K_\lambda)$ as the limit of $w(\Omega_\lambda^K\cap U_r)$ with $r\to +\infty$.

Now for $K\in \mathbb{Z}$ fixed, choose  $r\in\mathbb{Z}$ such that $r\ge K$. A simple Vitali's covering lemma can be applied now to $\Omega_\lambda^K\cap U_r$. We want to select a countable subfamily of disjoint base sets whose dilates cover $\Omega^K_\lambda\cap U_r$. More precisely, suppose that the set $\Omega^K_\lambda\cap U_r$ is nonempty. For each $x\in\Omega^K_\lambda\cap U_r$ there exists a base set $V_x$ of the form $V_x=y_x+U_{i_x}$ such that 
\begin{equation}\label{eq:simple-vitali-1}
\avgint_{V_x}|f(z)|\ d\mu>\lambda.
\end{equation}
Since we have that $i_x\le K$ for all $x\in \Omega^K_\lambda\cap U_r$, there is some $i_{1}=\max\{i_x\}$. We start the recursive selection procedure by picking one of this largest base sets as $V_1=y_1+U_{i_1}$. Now suppose that the first $V_1,V_2,\dots,V_k$ sets have been selected. We pick $V_{k+1}$ verifying that $V_{k+1}=y_{k+1}+U_{i_{k+1}}$ where 
$i_{k+1}=\max\{i_x: y_x+U_{i_x}\cap V_j=\emptyset, j=1,\dots, k\}$.

This process generates a sequence of disjoint base sets $\{V_k\}$. We note that the index sequence $\{i_k\}$ goes to $-\infty$ as $k$ goes to infinity. If not, since it is decreasing, there would be some $i_0=i_k$ for all $k\ge k_0$. Then we have that $ V_k\cap U_r\neq \emptyset$ and $i_k\le K\le r$ and by the engulfing property, $V_k\subset U_r^{**}$ for all $k\ge k_0$. In particular, the set $S=\{y_k: k\ge k_0\}\subset U_r^{**}$ is relatively compact. Then,  considering the set
\begin{equation*}
F=\bigcup_{k\ge k_0} V_k\subset S+U_{i_0}
\end{equation*}
and proceeding as in Lemma \ref{lem:CZ} we get a contradiction.

We claim now that 
\begin{equation*}\label{eq:simple-vitali-2}
\Omega_\lambda^K\cap U_r\subset \bigcup_{k\in \mathbb{N}} V_k^{**}.
\end{equation*}
To verify it, consider some $x\in \Omega^K_\lambda\cap U_r$ and the corresponding $V_x=y_x+U_{i_x}$. Suppose first that $V_x$ intersects some $V_k$. Let $k_0$ the smallest $k\in\mathbb{N}$ such that $V_x\cap V_k\neq\emptyset$. Then we have that $i_x\le i_{k_0}$, since $i_{k_0}$ was selected as the largest index among all the sets $V_x$ disjoint from $V_1,\dots, V_{k_0-1}$ (and by hypothesis $V_x$ is one of them). Then the engulfing property yields 
\begin{equation*}
V_x=y_x+U_{i_x}\subset y_{k_0}+U_{\theta^2(i_{k_0})}=V_{k_0}^{**}.
\end{equation*}
We are left to consider the case when $V_x\cap V_k=\emptyset$ for all $k\in \mathbb{N}$. But in this case, we would have that $i_x\le i_k$ for all $k$ and this is a contradiction since we saw that $i_k\to -\infty$.

Summing up, we find a countable collection of base sets $\{V_k\}_k$ such that 
\begin{equation*}
\avgint_{V_k}f\ d\mu >\lambda \qquad \text{ and } \qquad \Omega^K_\lambda\cap U_r\subset \bigcup_k V_k^{**}.
\end{equation*}

Then we can compute 
\begin{eqnarray*}
\lambda^q w(\Omega^K_\lambda\cap U_r)& \le & \sum_k \lambda^q w(V_k^{**})\\
& \le & \sum_k  w(V_k^{**})\left (\avgint_{V_k} w^{-\frac{1}{q}}w^{\frac{1}{q}}|f|\right )^q\\
& \le & \sum_k  \frac{w(V_k^{**})}{\mu(V_k)^q} \left (\int_{V_k}w^{1-q'}\ d\mu\right )^{q-1}  \left (\int_{V_k} |f|^qw\ d\mu\right )\\
& \le & D^{2q}\sum_k  \frac{w(V_k^{**})}{\mu(V_k^{**})^q} \left (\int_{V_k^{**}}w^{1-q'}\ d\mu\right )^{q-1}  \left (\int_{V_k} |f|^qw\ d\mu\right )\\
& \le & D^{2q}[w]_{A_q}\sum_k  \int_{V_k} |f|^qw\ d\mu\\
& \le & D^{2q}[w]_{A_q}\|f\|^q_{L^q_w}.
\end{eqnarray*}
From this estimate we conclude that
\begin{equation*}
\lambda^q w(\Omega_\lambda)\le D^{2q}[w]_{A_q}\|f\|^q_{L^q_w}
\end{equation*}
for any $\lambda >0$.

\end{proof}

Now we are able to present the proof of the sharp version of Buckley's Theorem for the maximal function $M$ on $L^p(G)$, $p>1$.

\begin{proof}[Proof of Theorem \ref{thm:CorSharpBuckley}:]

The idea is to use a sort of interpolation type argument, exploiting the sublinearity of the maximal operator $M$ and the weak  type estimate for $M$ from Lemma \ref{lem:weak}. For any $f\in L^p_w(G)$ and any $t>0$, define the truncation $f_t:=f\chi_{\{|f|>t\}}$. Then, an easy computation of the averages defining $M$ gives that 
\begin{equation*}
\{x\in G: Mf(x)>2t\}\subset \{x\in G: Mf_t(x)>t\}.
\end{equation*}

Now we compute the $L^p_w$ norm as follows
\begin{eqnarray*}
\|Mf\|^p_{L_w^p(G)} &=& \int_0^\infty pt^{p-1}w(\{x\in G: Mf(x)>t\}) dt\\
&=&2^p \int_0^\infty pt^{p-1}w(\{x\in G: Mf(x)>2t\}) dt\\
&\le &2^p \int_0^\infty pt^{p-1}w(\{x\in G: Mf_t(x)>t\}) dt.
\end{eqnarray*}

We recall the open property for Muckenhaupt weights: any $w\in A_p$ also belongs to $A_{p-\varepsilon}$ for some explicit $\varepsilon>0$ (see Lemma \ref{lem:OpenProperty}). 
Using the estimate of Lemma \ref{lem:weak} for $q=p-\varepsilon$, we obtain 

\begin{eqnarray}
\|Mf\|^p_{L_w^p(G)} &\leq&2^p p D^{2(p-\varepsilon)}[w]_{A_{p-\varepsilon}}\int_0^\infty t^{\varepsilon-1}\int_G f^{p-\varepsilon}_t(x) w(x)\ d\mu dt \nonumber \\
& = &  \frac{2^p p D^{2(p-\varepsilon)}[w]_{A_{p-\varepsilon}}}{\varepsilon}\int_G |f(x)|^p w\ d\mu \nonumber\\
& \leq & \frac{p2^{2p-1}D^{6p-2}[w]_{A_{p}}}{\varepsilon}\|f\|_{L^p_w(G)}^p, \label{eq:eqfinal}
\end{eqnarray}
where in the last inequality we have used Lemma \ref{lem:OpenProperty}.
Noticing that in Lemma  \ref{lem:OpenProperty}, $\varepsilon=\frac{p-1}{4D^{10}[\sigma]_{A_\infty}}$ we finally conclude from \eqref{eq:eqfinal} that 
\begin{equation*}
\|Mf\|_{L_w^p(G)} \le C \left ([w]_{A_p}[\sigma]_{A_\infty}\right )^\frac{1}{p}\|f\|_{L^p_w(G)}
\end{equation*}
and the proof of \eqref{eq:improvedBuckley} is complete. 

Finally, since $[\sigma]_{A_\infty}\leq[\sigma]_{A_{p'}}=[w]_{A_p}^{p'-1}$, \eqref{eq:Buckley2} follows from \eqref{eq:improvedBuckley}. 

\end{proof}

\bibliographystyle{amsalpha}

\providecommand{\bysame}{\leavevmode\hbox to3em{\hrulefill}\thinspace}
\providecommand{\MR}{\relax\ifhmode\unskip\space\fi MR }
\providecommand{\MRhref}[2]{%
  \href{http://www.ams.org/mathscinet-getitem?mr=#1}{#2}
}
\providecommand{\href}[2]{#2}

\end{document}